\documentclass[fontsize=11pt,DIV=12]{scrartcl}
\usepackage{lmodern} % Schrift: Latin Modern

\usepackage{epsf}
\usepackage[utf8]{inputenc}
\usepackage[T1]{fontenc}
\usepackage[intlimits]{amsmath}
\usepackage{amssymb}
\usepackage{amsthm}
\usepackage{amsfonts}
\usepackage{amscd}
\usepackage{yfonts}
\usepackage{bbm}
\usepackage{url}
\usepackage{dsfont}
\usepackage{marvosym}
\usepackage{color}
\newcommand{\red}{\color{black}}
\newcommand{\blue}{\color{black}}
\newcommand{\green}{\color{black}}

\allowdisplaybreaks[2] %Formeln d?rfen umgebrochen werden

\newtheorem{theorem}{Theorem}[section]
\newtheorem{lemma}[theorem]{Lemma}
\newtheorem{proposition}[theorem]{Proposition}
\newtheorem{corollary}[theorem]{Corollary}

\theoremstyle{definition}
\newtheorem{remark}[theorem]{Remark}

\DeclareMathOperator{\E}{\mathbb{E}}
\DeclareMathOperator{\Prob}{\mathbb{P}}

\renewcommand{\P}[2][]{\ensuremath{\mathbb{P}_{#1} \left( {#2} \right)}}
\newcommand{\Pkern}[1][]{\ensuremath{{}^{ #1}\!P}}

\newcommand{\Pfs}{\ensuremath{\mathbb{P}\text{-a.s.}}}
\newcommand{\Erw}[2][]{\ensuremath{\mathbb{E}_{#1} \left( {#2} \right)}}

\DeclareMathOperator{\N}{\mathbb{N}}

\DeclareMathOperator{\R}{\mathbb{R}}

\DeclareMathOperator{\C}{\mathbb{C}}

\renewcommand{\S}{\mathcal{S}}

\DeclareMathOperator{\llam}{\lambda\hspace{-5.1pt}\lambda}

\DeclareMathOperator{\supp}{\mathrm{supp}}

\renewcommand{\epsilon}{\varepsilon}
\renewcommand{\rho}{\varrho}

\newcommand{\esl}[1]{\ensuremath{\left( #1 \right)^\sim}}

\newcommand{\1}[1][]{\mathbf{1}_{#1}}
\newcommand{\norm}[1]{\ensuremath{\left\| {#1} \right\|}}
\newcommand{\abs}[1]{\ensuremath{\left| {#1} \right|}}
\newcommand{\skalar}[1]{\langle #1 \rangle}

\allowdisplaybreaks %Formeln d?rfen umgebrochen werden

\begin{document}

\title{Convergence to stable laws for multidimensional stochastic recursions: the case of regular matrices}

%\titlerunning{Convergence to stable laws: regular matrices}

%\author{Ewa Damek \thanks{E.~Damek was partially supported by  MNiSW grant N N201 393937, M.~Mirek was partially supported by MNiSW grant N N201 392337, J.~Zienkiewicz was partially supported by MNiSW grant N N201 397137, S. Mentemeier was supported by the Deutsche Forschungsgemeinschaft (SFB 878).} \and Sebastian Mentemeier \and Mariusz Mirek \and Jacek Zienkiewicz}

%%\institute{Ewa Damek \and Mariusz Mirek \and Jacek Zienkiewicz
%	\at University of Wroclaw, Institute of Mathematics, Grunwaldzki 2/4, 50-384 Wroclaw.
%	\and
%	Ewa Damek
%	\at \email{edamek@math.uni.wroc.pl}
%	\and
%	Mariusz Mirek
%	\at \email{mirek@math.uni.wroc.pl}
%	\and
%	Jacek Zienkiewicz
%	\at \email{zenek@math.uni.wroc.pl}
%	\and
%	Sebastian Mentemeier (\Letter)
%	\at University of Muenster, Institut f\"ur Mathematische Statistik, Einsteinstra\ss e 62, 48149 M\"unster.
%	\email{mentemeier@uni-muenster.de}}

\author{Ewa Damek\thanks{University of Wroclaw, Institute of Mathematics, Grunwaldzki 2/4, 50-384 Wroclaw. E.~Damek was partially supported by  MNiSW grant N N201 393937, e-mail: edamek@math.uni.wroc.pl, M.~Mirek was partially supported by MNiSW grant N N201 392337, e-mail: mirek@math.uni.wroc.pl, J.~Zienkiewicz was partially supported by MNiSW grant N N201 397137, e-mail zenek@math.uni.wroc.pl.}, Sebastian Mentemeier\thanks{University of Muenster, Institut f\"ur Mathematische Statistik,  Einsteinstra\ss e 62, 48149 M\"unster.  Research supported by the Deutsche Forschungsgemeinschaft (SFB 878). e-mail: mentemeier@uni-muenster.de}, Mariusz Mirek$^*$\hspace{-.15cm}, Jacek Zienkiewicz$^*$\\ \\ }

\maketitle

\begin{abstract}
Given a sequence $(M_{n},Q_{n})_{n\ge 1}$ of i.i.d.\ random variables with generic copy $(M,Q) \in GL(d, \R) \times \R^d$ , we consider the random difference equation (RDE)
\begin{equation*} \label{RDE:abstract} R_{n}=M_{n}R_{n-1}+Q_{n},\end{equation*} $n\ge 1$, and assume the existence of $\kappa >0$ such that
$$ \lim_{n \to \infty} \left(\E{\norm{M_1 \cdots M_n}^\kappa}\right)^{\frac{1}{n}} = 1 .$$
We prove, under suitable assumptions, that the sequence $S_n = R_1 + \dots + R_n$, appropriately normalized, converges in law to a multidimensional stable distribution with index $\kappa$. As a by-product, we show that the unique stationary solution $R$ of the RDE is regularly varying with index $\kappa$, and give a precise description  of its tail measure. %This extends prior work by Alsmeyer and the second author in \cite{AM2010}.
\bigskip


\end{abstract}

\section{Introduction}

Let $(M_{n},Q_{n})_{n\ge 1}$ be a sequence of i.i.d.\ random variables with generic copy $(M,Q)$ such that $M$ is a real $d\times d$ matrix and $Q$ takes values in $\R^{d}$. Suppose further that
\begin{equation}\label{logmom M}\tag{A1}
\E\log^{+}\|M\|<\infty
\end{equation}
where $\|M\|:=\sup_{|x|=1}|xM|$. Then, with $\Pi_{n}:=M_{1}\cdot...\cdot M_{n}$, there exists $\beta \in [-\infty, \infty)$ such that
\begin{equation*} %\label{def Liapunov exponent}
\beta:=\lim_{n\to\infty}n^{-1}\log\|\Pi_{n}\|\quad\Pfs
\end{equation*}
and defines the Liapunov exponent of the RDE
\begin{equation}\label{RDE}
R_{n}=M_{n}R_{n-1}+Q_{n},\quad n\ge 1.
\end{equation}
If $\beta$ is negative and
\begin{equation}\label{logmom Q}\tag{A2}
\E\log^{+}\|Q\|<\infty,
\end{equation}
then this recursive Markov chain has a unique stationary distribution which is given by the law of the almost surely convergent series
\begin{equation}\label{solution to RDE}
R:=\sum_{n\ge 1}\Pi_{n-1}Q_{n}.
\end{equation}This by now standard result may easily be deduced from a more general one for iterations of random Lipschitz maps,
{\blue see e.g. \cite{BP1992}} or \cite{DiacFreed1999} and references given there. Notice
that $R$ can also be characterised as the unique solution to  the
associated stochastic fixed point equation (SFPE),
\begin{equation} \label{SFPE} R \stackrel{d}= MR +Q, \end{equation} where $\stackrel{d}=$ denotes equality in distribution.

%Initiated by Kesten in his famous article \cite{Kesten1973}, the tail behaviour of $R$ has attained much interest,  for it can be considered as the stationary distribution of an ARCH process (see e.g. \cite{HRRV1989, Klueppelberg2004}) and appears in many mathematical or physical models (see the references in \cite{Kesten1973}).

Asymptotic properties of $R$ and $R_n$ were first investigated by
Kesten in his seminal paper \cite{Kesten1973}, and they are of
interest in various fields, e.g. financial time series or ARCH
processes that can be considered as a special case of \eqref{RDE},
see \cite{HRRV1989}, \cite{Klueppelberg2004}.
%We point out that our results can be extended to proof a CLT in the situation of \cite{Klueppelberg2004}.
Under suitable assumptions, the asymptotic behavior of $R$ is
governed by the value $\kappa >0$, which is the unique positive
solution to the equation
$$ \lim_{n \to \infty}  \left( \E{\norm{\Pi_n}^\kappa} \right)^{\frac{1}{n}} = 1 .$$

If $M$ is restricted to the set of similarities, this condition
simplifies to $\E\norm{M}^\kappa=1$. Then, Buraczweski et al.
showed that $R$ is regular varying \cite{BDGHU2009}, and that the
sequence \makebox{$S_n := R_1 + \dots + R_n$}, properly
normalized, converges in law to a multidimensional stable
distribution with index $\kappa$ \cite{BDG2010}. See also
\cite{Mirek2010} for a generalization to Lipschitz recursions.

In the present article, we apply these methods, {\red which were developed by Guivarc'h and Le Page in \cite{GL2008}}, to the case where
$M_n$ takes values in $GL(d,R)$, and its distribution powers
satisfy some density and irreducibility assumptions (stated below)
which were considered in \cite{AM2010}. The major ingredient is a
theorem by Basrak et al. \cite{BDM2002} that links regular
variation of $R$ with regular variation of its components
$\skalar{x,R}$. For related results on the tail behavior of
solutions to \eqref{RDE} under varying assumptions on the
distribution of $M$, see \cite{deSaporta2004},{\red\cite{GL2012}}, \cite{LePage1983}, as well as \cite{Goldie1991} for the case $d=1$.

In Section \ref{sect:prior:results} we review the results of
\cite{AM2010} and \cite{BDM2002}. Then we state our main results
in Section \ref{sect:statement}. The proofs are given in Sections
\ref{sect:tail:measure} to \ref{nondeg:section}.

{\red After submission of this paper, we learned about the recent preprint by Gao et al. \cite{GGL2011}. They use more intricate analysis to give similar results under more general conditions and in the case $\kappa =2$.%, namely (i-p)-condition as introduced in \cite{Gui2006,GL2012}.
}

\section{Prior results}\label{sect:prior:results}
%\subsection{Tail properties}
The assumptions on $(M,Q)$ considered here were first stated by
Kesten \cite[Theorem 6]{Kesten1973}.  {\red The proof of the
regular behavior of $R$ (in a more general setting) was given by Le Page \cite{LePage1983}, see also the recent preprint \cite{GL2012}.
 Using Kesten's assumptions, a short proof was recently given by Alsmeyer and the second
author by use of a different Markov Renewal Theorem \cite{AM2010}.}
Here we continue their investigations.

 We start with recalling
the main result of \cite{AM2010}. Given $x \in V:= \R^d\setminus\{0\}$, we write
$x^{\sim}$ for its projection on the unit sphere $S:=S^{d-1}$,
thus $x^{\sim} := |x|^{-1} x$. We will consider both column and
row vectors, without special mention as long as this is clear from
the context. Only scalar products are always understood to be
between column vectors. The Lebesgue measure on the space of real
$d \times d$-matrices, seen as $\R^{d^2}$, is denoted as
$\llam^{d^2}$ and the Lebesgue measure on $\R_+$ as $\llam$.
Finally, the open $\delta$-balls in $S$ and $GL(d,\R)$ with
centers $x$ and $A$ are denoted as $B_{\delta}(x)$ and
$B_{\delta}(A)$, respectively.

\begin{theorem}[Theorem 1.1 in \cite{AM2010}] \label{theorem:old}\ \\
Consider the RDE \eqref{RDE} and suppose that, in addition to \eqref{logmom M}, \eqref{logmom Q} and $\beta<0$, the following assumptions hold:
\begin{align}
&\Prob(M\in GL(d,\R))=1.\label{A3}\tag{A3}\\
&\max_{n\ge 1}\,\Prob((v\Pi_{n})^{\sim}\in U)>0\text{ for any $v\in S$ and any open $\emptyset\ne U\subset S$}.\label{A4}\tag{A4}\\
&\Prob(\Pi_{n_{0}}\in\cdot)\ge\gamma_{0}\1[{B_{c}(\Gamma_{0})}]\llam^{d^2}\text{ for some $\Gamma_{0}\in GL(d,\R)$, $n_{0}\in\N$ and $c,\gamma_{0}>0$}.\label{A5}\tag{A5}\\
%&\text{There exists a feasible matrix}.\label{A6}\tag{A6}\\
&\Prob(Mr+Q=r)<1\text{ for any column vector }r\in\R^{d}.\label{A6}\tag{A6}\\
&\text{There exists $\kappa_{0}>0$ such that}\notag\\
&\quad\E\inf_{v\in S}|vM|^{\kappa_{0}} \geq 1,\ \E\|M\|^{\kappa_{0}}\log^{+}\|M\|<\infty\text{ and }0<\E\|Q\|^{\kappa_{0}}<\infty.\label{A7}\tag{A7}
%&\quad\E\lambda_d (M^{\top}M)^{\kappa_{0}/2} \geq 1,\ \E\|M\|^{\kappa_{0}}\log^{+}\|M\|<\infty\text{ and }0<\E\|Q\|^{\kappa_{0}}<\infty.\label{A7}\tag{A7}
\end{align}
Then there exists a unique $\kappa \in (0, \kappa_0]$ such that
\begin{equation}\label{theorem:main:assertion1}
    \rho(\kappa) :=\lim_{n \to \infty}  \left( \E{\norm{\Pi_n}^\kappa} \right)^{\frac{1}{n}} = 1,
\end{equation}
and
\begin{equation}\label{theorem:main:assertion2}
\lim_{t \to 0} t^{-\kappa}\,\P{v(tR) > 1} = K(v)\quad\text{for all }v\in S,
\end{equation}
where $K$ is a finite positive and continuous function on $S$.
\end{theorem}

\begin{remark}
Assumption \eqref{A4} holds in particular if $\Gamma_0 = Id$ in \eqref{A5}. If
$M$ and $Q$ are independent, then \eqref{A6} holds trivially.
Regarding the condition $\E\inf_{v \in S} \abs{vM}^{\kappa_0}$, it
is shown in \cite[Section 5.1]{AM2010} that the function
$\rho(\varkappa)$ is log-convex and $\rho(\delta) < 1$ for some
$\delta >0$. Then this condition asserts that at least
$\rho(\kappa_0) \geq 1$, and so it implies the existence of
$\kappa$.  This is the only place it is needed and so it may
be replaced by any other condition assuring the existence of
$\kappa_0$ such that $\rho(\kappa_0)\geq 1$, compare \cite[Remark
2.8 (iii)]{Klueppelberg2004}. Our theorem is then applicable to the situation of
 \cite{Klueppelberg2004}.
\end{remark}

Next, we note an explicit description of the function $K(x)$. Therefore, we introduce the  operator on continuous functions on the sphere
\begin{equation} T_{\kappa}: \mathcal{C}(S) \to \mathcal{C}(S), \ \  T_\kappa f(v) := \Erw{f(\esl{vM}) \abs{vM}^\kappa}, \label{Tkappa0} \end{equation}
which was studied in \cite[Section 5.1]{AM2010}. The {\red operator
$T_\kappa$ is quasi-compact}, its spectral radius is given by
$\rho(\kappa)=1$, the only eigenvalue with modulus one is 1, and
the corresponding eigenspace is one-dimensional.

\begin{proposition}\label{prop!}
The function $K(v)$ in \eqref{theorem:main:assertion2} is given by
\begin{equation} \label{K} 0 < K(v) = \frac{r(v)}{\alpha \kappa} \int_{S} \frac{1}{r({y})}  \Erw{((yR)^+)^{\kappa}-((yMR)^+)^{\kappa}} \pi(dy) < \infty. \end{equation}
If $\eta$ is the unique distribution such that $\eta T_\kappa =
\eta$, then $r$ is the unique function with the property $T_\kappa
r = r$ and $\int_S r(v) \eta(dv) =1$. Finally, the distribution
$\pi$ is given by $\pi(dv)=r(v) \eta (dv)$, and
$$\alpha = \int_S \frac{1}{r(y)} \Erw{\log{\abs{yM}} r(\esl{yM}) \abs{yM}^\kappa} \pi(dy).$$
\end{proposition}

\begin{proof}
The formula for $K(v)$ is given in \cite[Lemmas 6.1 and
6.4]{AM2010}. The function $r$ is defined by $T_\kappa r =r$ in
\cite[Lemma 5.4]{AM2010}. Then $\pi$ is defined as the unique
stationary distribution of the Markov chain with transition kernel
${}^\kappa\Pkern$ given by (see \cite[(18)]{AM2010})
${}^\kappa\Pkern f(v) := {r(v)}^{-1} T_\kappa (f \cdot r)(v) $.
Thus
$$ \int_S \bigl(f(v)\cdot r(v)\bigr) \frac{\pi(dv)}{r(v)} = \int_S T_\kappa(f \cdot r) \frac{\pi(dv)}{r(v)} .$$
Since $r$ is continuous and positive, any function $g \in \mathcal{C}(S)$ can be written as $g=f \cdot r$, $ f \in \mathcal{C}(S)$, thus $(r^{-1} \pi) T_\kappa = (r^{-1} \pi)$, and by uniqueness, $\pi = r \eta$. The expression for $\alpha$ may be found in \cite[Lemma 5.9]{AM2010}. 
\end{proof}

%\subsection{Regular variation}

The random variable $R$ is said to be regularly varying with index $\kappa$, if there exists a slowly varying function $L$ on $\R_+$ and a Radon measure $\Lambda_\kappa$ on $V$ (called tail measure), such that for all compactly supported continuous $f$ (i.e. $f \in \mathcal{C}_c(V)$)
\begin{equation}
\label{mdreg:var} \lim_{t \to 0} t^{-\kappa} L(t^{-1}) \Erw{f(tR)} = \int_V f(x) \Lambda_\kappa(dx) .
\end{equation}

Basrak et. al. \cite{BDM2002} investigated conditions under which \eqref{theorem:main:assertion2}  already implies that $R$ is regularly varying with index $\kappa$. For noninteger $\kappa$, this holds true, see \cite[Theorem 1.1 (ii)]{BDM2002}. If $\kappa$ is an odd integer, a close inspection of their proof shows that it still goes through, provided the distribution of $R$ is symmetric, i.e.  $\P{R \in \cdot} = \P{-R \in \cdot}$.
This can also be deduced from \cite[Corollary 2]{BL2009} combined with a comment after the proof of \cite[Theorem 3b]{BL2009}. Both \cite{BDM2002}, \cite{BL2009} also give counterexamples, showing that the condition $\kappa \notin \N$ in general may not be omitted.

For future reference, we formulate this result as a proposition (see also \cite[Theorem 4.3]{Klueppelberg2007}).
Note that by \eqref{solution to RDE}, a sufficient condition for
$R$ to have a symmetric distribution is that Q has a symmetric
distribution.

\begin{proposition}\label{prop:reg:var}
Under the assumptions of Theorem \ref{theorem:old}, if
\begin{itemize}
\item $\kappa \notin \N$, or
\item $\kappa$ odd, and $Q$ has symmetric distribution,
\end{itemize}
then $R$ is regularly varying with index $\kappa$. The function $L$ is equal to $\1$.
\end{proposition}

\section{Statement of results}\label{sect:statement}

Our first theorem gives more precise information about the tail measure, and extends the class of test functions in \eqref{mdreg:var}. Its proof will be given in section \ref{sect:tail:measure}. By $T_\kappa^*$ we denote the operator
\begin{equation} \label{Tkappa2}
 T_{\kappa}^*: \mathcal{C}(S) \to \mathcal{C}(S), \ \  T_\kappa^* f(v) := \Erw{f(\esl{Mv}) \abs{Mv}^\kappa}.
\end{equation}

\begin{theorem}\label{theorem:vague:convergence}
{\red Under the assumptions of \ref{theorem:old}, let $\kappa \notin \N$ or $\kappa$ odd and $Q$ symmetric.} Then the tail measure $\Lambda_\kappa$ is a product measure $\Lambda_\kappa= \sigma_\kappa \otimes \llam_\kappa$ on {\blue $V= \R^d\setminus\{0\}  $,} where $\sigma_\kappa$ is a finite nonzero measure satisfying  $\sigma_\kappa T_\kappa^* = \sigma_\kappa$, and $\llam_\kappa(ds) :={s^{-(\kappa+1)}} \llam(ds)$.
Moreover, for every $\Lambda_\kappa$-a.e. continuous function $f$, such that
\begin{equation} \sup_{x \in {V}} \abs{x}^{-\kappa} \abs{\log \abs{x}}^{1+\epsilon} \abs{f(x)} < \infty \label{prop:fkonv} \end{equation}
for some $\epsilon >0$, we have
\begin{equation} \lim_{t \to 0} t^{-\kappa} \Erw{f(tR)} = \int_{V} f(x) \Lambda_\kappa(dx) \left( = \int_0^\infty \int_S f(sv) \sigma_\kappa(dv) \frac{1}{s^{\kappa+1}} ds \right).    \label{conv:f} \end{equation}
\end{theorem}

{\red
\begin{remark}
It is interesting to compare this result with \cite[Theorem 1.10]{BDM2010}, which describes the tail measure in the case where the heavy-tail behaviour of $R$ is caused by heavy-tailed input $Q$.
\end{remark}

}

A distribution on $\R^d$ with characteristic function $\Xi$ is called stable, if for all $n \in \N$ there exist $\gamma_n >0$ and $x_n \in \R^d$, such that for all $y \in \R^d$
\begin{equation}
(\Xi(y))^n = \Xi(\gamma_n y) e^{i \skalar{x_n, y}} .
\end{equation}
This holds in particular true, if there is $\kappa >0$ and a function $C_\kappa : S \to \C$, such that for all $s \geq 0$, $v \in S$,
\begin{equation} \label{stable:law:form}
\Xi(sv) = \exp(s^\kappa C_\kappa(v)),
\end{equation}
since then
$$ \bigl(\Xi(sv)\bigr)^n = \exp(n s^\kappa C_\kappa(v)) = \exp( (n^{\frac{1}{\kappa}} s)^\kappa C_\kappa(v)) = \Xi\bigl(( n^{\frac{1}{\kappa}} s) v \bigr) .$$
A distribution on $\R^d$ is fully nondegenerate, i.e. its support is not contained in any lower dimensional (=proper) subspace of $\R^d$ iff the set $\{ y \in \R^d : \abs{\Xi(y)} < 1 \}$ is not contained in a proper subspace. This can easily be seen by considering all onedimensional marginal distributions. If $\Xi$ is of the form \eqref{stable:law:form}, then an equivalent condition is that the set $\{v\in S: \Re C_{\kappa}(v)<0\}$ is not contained in any proper subspace of $\mathbb{R}^d$.

Now we are ready to formulate our main theorem, which will be proved in sections \ref{KLT:section} and \ref{nondeg:section}.

\begin{theorem}\label{theorem:stable:laws}
Under the assumptions of Theorem \ref{theorem:old},  write $R_n^x$ for the $n$-th iteration of \eqref{RDE} started with $R_0=x$, and $S_n^x := \sum_{k=1}^n R_k^x$. Let $W(x)=\sum_{k=1}^\infty M_k\cdot \ldots\cdot M_1 x$, and $h_v(x)= \Erw{e^{i \skalar{v, W(x)}}}$ for $x \in \R^d$ and $v\in S$.  If $\kappa=1$, assume that the distribution of $Q$ is symmetric and set {\green $\xi(t) = \Erw{ \frac{tR}{1+ \abs{tR}^2} }$.}
\begin{itemize}
\item If $\kappa\in(0, 1)\cup(1, 2)$, then there is a sequence $d_n=d_n(\kappa)=nm_{\kappa}$ and a function
$C_{\kappa}:S\mapsto\mathbb{C}$ such that the random variables
$n^{-\frac{1}{\kappa}}\left(S_n^x-d_n\right)$ converge in law to the $\kappa$-stable random variable
with characteristic function of the form \eqref{stable:law:form}, where
\begin{align*}
C_{\kappa}(v)=\int_{{\red V }} \left(\left(e^{i \skalar{v,x}}-1\right) h_v(x)-i\mathbf{1}_{(1,2)}(\kappa)\skalar{v,x} \right) \Lambda_\kappa(dx),
\end{align*}
and $m_{\kappa}=\mathbf{1}_{(1,2)}(\kappa)\mathbb{E}R$.
\item  If $\kappa=1$, then there is a function $C_{1}:S\mapsto\mathbb{C}$ such that the random variables $n^{-1}S_n^x-n\xi(n^{-1})$ converge in
law to the 1-stable distribution with characteristic function of the form \eqref{stable:law:form}, where
\begin{align*}
    C_{1}(v)=\int_{{\red V }}\left(\left(e^{i\langle v, x\rangle}-1\right)h_v(x)-\frac{i\langle v, x\rangle}{1+|x|^2}\right)\Lambda_1(dx).
\end{align*}
\end{itemize}

Moreover, if \begin{equation}
\max_{n\ge 1}\,\Prob((vM_1^\top \cdots M_n^\top)^{\sim}\in U)>0\text{ for any $v\in S$ and any open $\emptyset\ne U\subset S$},\label{A4*}\tag{A4*}
\end{equation} holds,  then the limit laws are fully nondegenerate.
\end{theorem}

{\blue The case where $\kappa>2$ has been widely investigated in the  context of Lipschitz maps defined on complete separable metric spaces, see \cite{Be,WW} where the authors exploited martingale methods. On the other hand we refer to \cite{HH1,HH,HP} where, like in our case, the spectral methods play crucial roles.}

\begin{remark} \label{remark!2}
Note that $\eqref{A4}$ is the only assumption that changes when considering the transpose $M^\top$ instead of $M$. So using \eqref{A4*} instead of \eqref{A4}, Theorem \ref{theorem:old} holds for the solution $R^*$ of the SFPE
\begin{equation} \label{SFPE:transpose} R^* \stackrel{d}= M^\top R^* + Q,
\end{equation}
with the same $\kappa$ due to \eqref{theorem:main:assertion1}.
%or equivalently by
%\begin{equation} \label{Tkappa2}
%T_{\kappa}^*: \mathcal{C}(S) \to \mathcal{C}(S), \ \  f(x) \mapsto \Erw{f(\esl{Mx}) \abs{Mx}^\kappa},
%\end{equation}
%when considering column vectors instead of row vectors.
\end{remark}

{\green \begin{remark}
Similar results can be obtained in the case of nonnegative $(M,Q)$, satisfying the assumptions of Kesten \cite[Theorem B]{Kesten1973} - see also \cite{BDG2011,GL2012,Mirek2011} for weaker versions of these assumptions. The proof goes along the same lines, just starting with \cite[Theorem B]{Kesten1973} instead of our Theorem \ref{theorem:old}. The limit laws are supported in the positive cone if $\kappa <1$.
\end{remark} }

%% Mariusz version %%

%\begin{itemize}
%\item If $\kappa\in(0, 1)\cup(1, 2)$, then there is a sequence $d_n=d_n(\kappa)=nm_{\kappa}$ and a function
%$C_{\kappa}:S\mapsto\mathbb{C}$ such that the random variables
%$n^{-\frac{1}{\kappa}}\left(S_n^x-d_n\right)$ converge in law to the $\kappa$-stable random variable
%with characteristic function
%\begin{align*}
%\Xi_{\kappa}(tv)=\exp(t^{\kappa}C_{\kappa}(v)),\ \ \ \mbox{for $t>0$ and
%$v\in S$,}
%\end{align*}
%where
%\begin{align*}
%C_{\kappa}(v)=\int_{\R^d_*} \left(\left(e^{i \skalar{v,x}}-1\right) h_v(x)-i\mathbf{1}_{(1,2)}(\kappa)\skalar{v,x} \right) \Lambda(dx),
%\end{align*}
%and $m_{\kappa}=\mathbf{1}_{(1,2)}(\kappa)\mathbb{E}R$. Moreover, $C_{\kappa}(tv)=t^{\kappa}C_{\kappa}(v)$ for every $t>0$ and
%$v\in S$.
%\item  \textbf{$\Lambda$ and $\nu$ are symmetric!}If $\kappa=1$, then there are functions $\xi, \tau:(0, \infty)\mapsto\mathbb{R}$ and
%$C_{1}:S\mapsto\mathbb{C}$ such that the random variables $n^{-1}S_n^x-n\xi(n^{-1})$ converge in
%law to the random variable with characteristic function
%\begin{align*}
%\Xi_{1}(tv)=\exp(tC_{1}(v)+it\langle v, \tau(t)\rangle),\ \ \
%\mbox{for $t>0$ and $v\in S$,}
%\end{align*}
%where
%\begin{align*}
%    C_{1}(v)=\int_{\mathbb{R}^d_*}\left(\left(e^{i\langle v, x\rangle}-1\right)h_v(x)-\frac{i\langle v, x\rangle}{1+|x|^2}\right)\Lambda(dx),
%\end{align*}
%$\xi(t)=\mathbb{E}\left(\frac{tR}{1+|tR|^2}\right)$, and $\tau(t)=\int_{\mathbb{R}^d_*}\left(\frac{x}{1+|tx|^2}-\frac{x}{1+|x|^2}\right)\Lambda(dx)$.
%\end{itemize}

\section{On the tail measure} \label{sect:tail:measure}
Proposition \ref{prop:reg:var} yields the regular variation of $R$, given $\kappa \notin \N$, or $\kappa$ odd, and $Q$ having symmetric distribution. Property \eqref{mdreg:var} implies that $\Lambda_\kappa$ is a product measure on $V$ considered as $S \times \R_+$. {\blue We will show that
\begin{align}\label{polar}
\int_{V}f(x)\Lambda_{\kappa}(dx)=\int_{0}^{\infty}\int_{S}f(rx)\sigma_{\kappa}(dx)\frac{dr}{r^{\kappa+1}}.
\end{align}
Indeed, let $\Phi: V\mapsto(0, \infty)\times S$ be
defined as follows: $\Phi(x)=\left(|x|, \frac{x}{|x|}\right)$ and
its inverse $\Phi^{-1}:(0,
\infty)\times S\mapsto V$ by $\Phi^{-1}(r,
z)=rz$. Let us define the measure $\sigma_{\kappa}$ on  $S$
\begin{align*}
\sigma_{\kappa}(A):=\kappa\Lambda_{\kappa}\left(\Phi^{-1}\left[(1, \infty)\times A\right]\right),
\end{align*}
for $A\in\S$. Now on the one hand we have that
\begin{align*}
\int_{V}f(x)\Lambda_{\kappa}(dx)=\int_{(0, \infty)\times S}f(rz)\left(\Lambda_{\kappa}\circ\Phi^{-1}\right)(dr, dz).
\end{align*}
On the other hand for all $A \in \S$ and $s >0$ we have
\begin{multline*} 
\Lambda_\kappa\left(\Phi^{-1}[A \times (s, \infty)]\right) =  \Lambda_\kappa\left(s\Phi^{-1}[A \times (1, \infty)]\right)\\
= \lim_{t \to 0} t^{-\kappa} \P{tR \in s\Phi^{-1}[A \times (1, \infty)]}
=  s^{-\kappa} \lim_{t \to 0}  \left(\frac{t}s \right)^{-\kappa} \P{\frac{t}{s}R \in \Phi^{-1}[A \times (1, \infty)]}\\
=  s^{-\kappa} \Lambda_\kappa\left(\Phi^{-1}[A \times (1, \infty)]\right)  =  \kappa \int_s^\infty \Lambda_\kappa\left(\Phi^{-1}[A \times (1, \infty)]\right) \, t^{-\kappa-1} dt.
\end{multline*}
Thus $\Lambda_\kappa = \sigma_\kappa \otimes \llam_\kappa$, and formula \eqref{polar} follows.}

\begin{lemma}
The measure $\sigma_\kappa$, as defined above, is finite and nonzero.
\end{lemma}

\begin{proof}
The measure $\sigma_\kappa$ is nonzero, since by \eqref{theorem:main:assertion2}
\begin{align*}\sigma_\kappa(S) = \kappa\Lambda_\kappa((1, \infty)) = & \lim_{t \to 0} t^{-\kappa} \P{\abs{tR} \in (1, \infty)} \\ \geq & \lim_{t \to 0} t^{-\kappa} \P{\skalar{x,tR} \in (1, \infty)} = K(x) >0
\end{align*}
for all $x \in S$. On the other hand, observe that for $e_1, \dots, e_d$ being an orthonormal basis of $\R^d$,
$$ t^{-\kappa}\P{\abs{tR} > 1} \leq \sum_{i=1}^d \Bigl( t^{-\kappa}\P{\skalar{e_i, tR} > d^{-2}} + t^{-\kappa}\P{\skalar{-e_i, tR} > d^{-2}} \Bigr),$$
and the limit of the right hand side is finite again by
\eqref{theorem:main:assertion2}. 
\end{proof}

\begin{remark} \label{Lambda:tight}
This shows in particular, that for all $\epsilon >0$ there exists $C>0$ such that
$$ \Lambda_\kappa(S \times [C, \infty)) = \Lambda_\kappa(B_C(0)^c) < \epsilon.$$
\end{remark}
This is the main ingredient for showing that \eqref{conv:f} holds for the more general class of test functions given in \eqref{prop:fkonv}. The proof goes along the same lines as the proof of Theorem 2.8 in \cite{BDGHU2009} and is therefore omitted.

The next proposition characterizes $\Lambda_\kappa$ as a stationary measure of the Markov Chain on $V$, given by the action of $M$ on $V$. %\emph{We pose it as an open question to describe all stationary (Radon) measures of this chain.}
\begin{proposition}
For the tail measure $\Lambda_\kappa$, and all $f \in \mathcal{C}_c(V)$,
\begin{equation} \int_{V} f(x) \Lambda_\kappa(dx) = \int_{V} \Erw{f(Mx)} \Lambda_\kappa(dx) . \label{invariant:measure}
\end{equation}
\end{proposition}

\begin{proof}
We introduce the space of $\epsilon$-H\"older continuous functions, defined by
\[ H^\epsilon = \{ f \in \mathcal{C}_c(V) \ : \ \sup_{x,y \in {{\red V }}} \frac{\abs{f(x)-f(y)}}{\abs{x-y}^\epsilon} < \infty \} .\]
One can proceed as in the proof of Lemma 2.19 in \cite{BDGHU2009}, to show that if $0 < \epsilon < \min\{1,\kappa\}$, then for all  $f \in H^\epsilon$
\[ \lim_{t \to 0} t^{-\kappa} \Erw{f(tR) - f(tMR)} = 0.  \]
Since $\lim_{t \to 0} t^{-\kappa} \Erw{f(tR)}$ exists, this yields the existence of
$ \lim_{t \to 0} t^{-\kappa} \Erw{f(tMR)}$, and it is equal to $\int_V f(x) \Lambda_\kappa(dx)$.   It remains to show that the limit is also equal to $\int_V \Erw{f(Mx)} \Lambda_\kappa(dx)$ - note that $x \mapsto \Erw{f(Mx)}$ in general has unbounded support. {\red Let $\supp f \subset \overline{B_\eta(0)}^c, \eta >0$. Then for any $t >0$ and $m \in GL(\R,d)$,
\begin{align*}
& t^{-\kappa} \Erw{f(tmR)} \le t^{-\kappa} \norm{f}_\infty \Erw{\1[{\{\abs{tmR} > \eta\}}]} \\
\le &  \norm{m}^\kappa \norm{f}_\infty  \Bigl( t^{-\kappa} \norm{m}^{-\kappa} \P{t \norm{m} \abs{R} > \eta} \Bigr) \\
& \le  \norm{m}^\kappa \norm{f}_\infty \sup_{s >0} s^{-\kappa} \P{s\abs{R}> \eta},
\end{align*}
and as above, using \eqref{theorem:main:assertion2}, $C:=\sup_{s >0} s^{-\kappa} \P{s\abs{R}> \eta} < \infty$, thus
$$ \int \Bigl( \norm{m}^\kappa \norm{f}_\infty \sup_{s >0} s^{-\kappa} \P{s\abs{R}> \eta} \Bigr) \P{M \in dm}= C \norm{f}_\infty \Erw{\norm{M}^\kappa} < \infty.$$
Then we may use the dominated convergence theorem and that  for fixed $m \in GL(d, \R)$, $x \mapsto f(mx)$ has compact support,  to infer
\begin{align*} & \lim_{t \to 0} t^{-\kappa} \Erw{f(tMR)} = \int \lim_{t \to 0} t^{-\kappa} \Erw{f(tmR)} \P{M \in dm} \\ = & \int \int_V f(mx) \Lambda_\kappa(dx) \P{M \in dm} = \int_V \Erw{f(Mx)} \Lambda_\kappa(dx).
\end{align*}}
Finally observe that $H^\epsilon$ is dense in $\mathcal{C}_c(V)$ due to the Stone-Weierstrass theorem, so we infer the assumption for all $f \in \mathcal{C}_c(V)$. 
\end{proof}

\begin{lemma}\label{lemma:sigma}
For the operator $T_\kappa^*$ introduced in \eqref{Tkappa2}, $\sigma_\kappa T_\kappa^*  = \sigma_\kappa.$
\end{lemma}

\begin{proof}
First notice that due to Remark \ref{Lambda:tight}, the identity \eqref{invariant:measure} also holds for bounded continuous functions $f$ on $V$ such that $\supp f \cap B_\eta(0) = \emptyset$ for some $\eta >0$, in particular for functions $g_u(sv)=f(v) \1[{(u, \infty)}](s)$, where $f$ is any continuous function on $S$, and $u >0$. Then
\begin{align*}
 & \int_u^\infty \int_S f(v) \sigma_\kappa(dv) \frac{1}{s^{\kappa +1}} ds = \int_0^\infty \int_S \Erw{f(\esl{Mv})\1[(u, \infty)](s\abs{Mv})} \sigma_\kappa(dv) \frac{1}{s^{\kappa+1}} ds \\
= & \Erw{  \int_S \int_0^\infty f(\esl{Mv}) \1[(u, \infty)](t) \frac{\abs{Mv}^\kappa}{t^{\kappa+1}} dt \sigma_\kappa(dv) }\\
= & \int_u^\infty \int_S \Erw{f(\esl{Mv}) \abs{Mv}^\kappa} \sigma_\kappa(dv) \frac{1}{t^{\kappa+1}} dt .
\end{align*}
Since $u$ is arbitrary, we infer $\int_S f(v) \sigma_\kappa (dv) = \int_S \Erw{f(\esl{Mv}) \abs{Mv}^\kappa} \sigma_\kappa(dv)$ for all $f \in \mathcal{C}(S)$. 
\end{proof}

\section{Convergence to stable laws}\label{KLT:section}

{\red The proof of Theorem \ref{theorem:stable:laws} is based on the spectral method initiated by Nagaev \cite{Nagaev1957} and its first application in this context due to Guivarc'h and Le Page \cite{GL2008}}.  For $t >0$, $v \in S$, we consider Fourier perturbations
$$ \Pkern_{t,v} f(x) = \Erw{e^{i\skalar{tv, Mx+Q}} f(Mx+Q)}=\Erw{e^{i\skalar{tv, R_1^x}} f(R_1^x)}, $$ of $P$, the transition operator of the Markov chain corresponding to \eqref{RDE}. Note that $\Pkern_{0,v} = \Pkern$ for all $v \in S$. Since for $n \geq 1$,
$$ \Pkern^n_{t,v} f(x) = \Erw{e^{i\skalar{tv, S_n^x}}f(R_n^x)}, $$
we are interested in the asymptotics of $\Pkern^n_{t,v}\1$.

The main tool used here is the Keller-Liverani theorem
\cite{KL1999}, which links the properties of $\Pkern$ and
$\Pkern_{t,v}$ on appropriate function spaces for small $t$. The
proof goes along the same lines as in \cite{BDG2010},
\cite{Mirek2010}, so we only give an outline. In particular we do
not check the assumptions of the theorem, which can be copied from
\cite{BDG2010}) for the case $\kappa\in(0, 1)\cup(1, 2)$. The case
when $\kappa=1$ is more technical, the detailed exposition is
contained in \cite{BDM2010}.
%\textbf{Here is the link to this paper} \verb"http://arxiv.org/PS_cache/arxiv/pdf/1011/1011.1685v1.pdf"

First observe that the unique eigenvalue of $\Pkern$, acting on
$\mathcal{C}(\R^d)$, of modulus one is 1, since the corresponding
Markov chain has a unique stationary distribution $\nu = \P{R \in
\cdot}$. On an appropriately chosen function space $\mathcal{B}$
containing $\1$, $\Pkern$ also has a spectral gap, see \cite[3.13
and 3.14]{BDG2010}. Thus $\Pkern$ is quasi-compact, i.e. we have
the decomposition
$$ \Pkern =1 \cdot \Pi  + Q ,$$
where $\Pi$ is a one dimensional projection given by $\Pi f = \nu(f) \1$, the spectral radius of $Q$ is strictly smaller than 1, and both operators commute.

{\red It is now a result of the Keller-Liverani-theorem that, for small t, the operators $\Pkern_{t,v}$, acting on $\mathcal{B}$, are also quasi-compact} with singular dominant eigenvalue $k(t, v)$, such that for all $n \ge 1$
$$ \Pkern_{t,v}^n = k(t,v)^n \Pi_{t,v} + Q_{t,v}^n, $$
where again $\Pi_{t,v}$ is a onedimensional projection operator commuting with $Q_{t,v}$, and the spectral radii of $Q_{t,v}$ are uniformly bounded by some $\rho <1$. Moreover, $k(t,v) \to 1$, as well as $\Pi_{t,v} \to \Pi$, $Q_{t,v} \to Q$ as operators on $\mathcal{B}$, see \cite[3.17 and 3.18]{BDG2010}. The following lemma provides the link to the characteristic functions $$\Xi_{\kappa,n}(sv) = \Erw{e^{i\left\langle sv,  n^{-1/\kappa}(S_n^x-nm_{\kappa})\right\rangle}}$$ of the normalized Birkhoff sums.

\begin{lemma}
Under the assumptions of Theorem \ref{theorem:stable:laws},
with  $t_n:=sn^{-1/\kappa}$ for $n\in\mathbb{N}$, $s>0$ and $v\in S$ we have (provided the right hand side exists)
{\red $$  \lim_{n \to \infty} \Xi_{\kappa,n}(sv) =  \exp\left( s^{\kappa}\cdot\lim_{n \to \infty}\frac{k(t_n, v)-1-i\skalar{v,t_nm_{\kappa}}}{t_n^\kappa} \right) .$$ }
\end{lemma}

%{\red \begin{proof}
%Using that $\lim_{t \to 0} \Pi_{t, v} \1 (x)= 1$ for all $x \in
%\R^d$ and $r\left(Q_{t,v}\right)\le\rho < 1$, for all $v \in S$,
%and $t$ small enough, we see
% that $\lim_{n \to \infty} \Xi_{\kappa,n}(sv)$ equals
% \begin{align*}
% &  \Erw{e^{it_n\left\langle v,  S_n^x-nm_{\kappa}\right\rangle}}=  e^{-int_n\langle v, m_{\kappa}\rangle}\Erw{e^{it_n\left\langle v,  S_n^x\right\rangle}} \\
% = &  e^{-int_n\langle v, m_{\kappa}\rangle}\cdot(P^n_{t_n, v} \1 )(x) \\
% = & e^{-int_n\langle v, m_{\kappa}\rangle}\cdot\left( k^n(t_n, v) (\Pi_{t_n, v} \1) (x) + (Q^n_{t_n, v} \1)(x)\right) \\  = & e^{-int_n\langle v, m_{\kappa}\rangle}\cdot k^n(t_n, v) \cdot (1 + o(1)) + O(\rho^n).
%\end{align*}
%Then
%\begin{align}
%\log \Xi_{\kappa,n}(tv) = & n \left( -it_n \langle v, m_\kappa \rangle + \log k(t_n,v) \right) +o(1) \nonumber\\
%= & n \left( -it_n \langle v, m_\kappa \rangle + k(t_n,v) -1 + O\Bigl(k(t_n,v)-1)^2\Bigr) \right) +o(1) \label{oterm}
%\end{align}
%Since $m_\kappa = 0$ if $\kappa <1$, resp. $t_n = O(t/n)$ if $\kappa \ge 1$, we have that if
%\begin{equation} n \Bigl( -it_n \langle v, m_\kappa \rangle + k(t_n,v) -1) \Bigr) = s^\kappa \cdot \frac{k(t_n,v) -1 -i \langle v, t_n m_\kappa \rangle}{t_n^\kappa} \label{limit}
%\end{equation}
%converges for $n \to \infty$, then $\Bigl(k(t_n,v)-1\Bigr)=O(1/n)$. Thus given the limit in \eqref{limit} exists, we have
% $$ \log \Xi_{\kappa,n}(tv) = s^\kappa \cdot \frac{k(t_n,v) -1 -i \langle v, t_n m_\kappa \rangle}{t_n^\kappa} +  o(1) . $$\qed
%\end{proof} }

\begin{proof}
Using that $\lim_{s \to 0} \Pi_{s, v} \1 (x)= 1$ for all $x \in
\R^d$ and $r\left(Q_{s,v}\right)\le\rho < 1$, for all $v \in S$,
and $s$ small enough, we see
 that $\lim_{n \to \infty} \Xi_{\kappa,n}(tv)$ equals
 \begin{align*}
 & \lim_{n \to \infty} \Erw{e^{it_n\left\langle v,  S_n^x-nm_{\kappa}\right\rangle}}= \lim_{n \to \infty} e^{-int_n\langle v, m_{\kappa}\rangle}\Erw{e^{it_n\left\langle v,  S_n^x\right\rangle}} \\
 = & \lim_{n \to \infty} e^{-int_n\langle v, m_{\kappa}\rangle}\cdot(P^n_{t_n, v} \1 )(x) \\
 = & \lim_{n \to \infty} e^{-int_n\langle v, m_{\kappa}\rangle}\cdot\left( k^n(t_n, v) (\Pi_{t_n, v} \1) (x) + (Q^n_{t_n, v} \1)(x)\right) \\  = & \lim_{n \to \infty} e^{-int_n\langle v, m_{\kappa}\rangle}\cdot k^n(t_n, v).
\end{align*}
Assume for the moment, that $ n \cdot (e^{-it_n\langle v, m_{\kappa}\rangle}\cdot k(t_n, v)-1) $ {\red has a    limit}  as $n$ goes to infinity (this will be shown later).  We infer
\begin{align}
\lim_{n \to \infty} e^{-int_n\langle v, m_{\kappa}\rangle}\cdot k^n(t_n, v) &
 = \lim_{n \to \infty} (1 + \frac1n n \cdot(e^{-it_n\langle v, m_{\kappa}\rangle}\cdot k(t_n, v) -1))^n \nonumber  \\
 & = \exp\left( \lim_{n \to \infty} n \cdot (e^{-it_n\langle v, m_{\kappa}\rangle}\cdot k(t_n, v)-1) \right) \nonumber \\
& = \exp \left( t^{\kappa}\cdot\lim_{n \to \infty}\frac{k(t_n, v)-1-i\skalar{v,t_nm_{\kappa}}}{t_n^\kappa} \right). \label{expansion:ev}
\end{align}
The last line follows from the definition of $t_n$ and $m_\kappa=0$ if $\kappa \in (0,1)$. If $\kappa \in (1,2)$, we expand
$$ n \cdot (e^{-it_n\langle v, m_{\kappa}\rangle}\cdot k(t_n, v)-1) = n \cdot \Bigl( \bigl[1-it_n\skalar{ v, m_{\kappa}}+ o(1/n)\bigr]\cdot \bigl[ (k(t_n, v)-1) +1\bigr]-1\Bigr).$$
In particular, if the limit in \eqref{expansion:ev} exists, also the assumption made on the way is true, and our calculation was justified. 
\end{proof}

%%%%%% end of former version

Thus we have to study the behaviour of $t^{-\kappa} \bigl( k(t,v) - 1 - i\skalar{v,tm_\kappa} \bigr)$ as $t$ tends to zero.

Let $g_{t,v}$ be an eigenfunction of $P_{t,v}$, corresponding to the dominant eigenvalue $k(t,v)$.
From $\Pkern_{t,v} g_{t,v} = \Pkern e^{it\skalar{v, \cdot}} g_{t,v}$ and $\nu \Pkern = \nu$ it follows that
\begin{align}(k(t,v)-1) \nu(g_{t,v})= \int_{\mathbb{R}^d} (e^{it \skalar{v,x}} -1 ) g_{t,v}(x) \nu(dx). \label{identity:ktv}
\end{align}

Now we need an explicit expression for $g_{t,v}$. It turns out that it is closely related to $h_v$: Denote $\Delta_t f(x) = f(tx)$ for $t>0$. Introducing auxiliary operators
$$ T_{t,v} f(x) = \Erw{e^{i \skalar{v, Mx+tQ}}f(Mx+tQ)} ,$$  we have the identity
\begin{equation} T_{t,v} = \Delta_t^{-1} \circ P_{t,v} \circ \Delta_t . \label{intertwine} \end{equation}
This relates the eigenvalues of both operators for $t>0$. Namely, if $f$ is an eigenfunction of $T_{t,v}$, then $\Delta_t f$ is an eigenfunction of $P_{t,v}$. So it is not surprising that the Keller-Liverani-Theorem holds for the operators $T_{t,v}$, yielding a decomposition
$$ T_{t,v} = k(t,v) \Pi_{T,t,v}  + Q_{T,t,v} ,$$
with $k(t,v) \to k(0,v)$, $\Pi_{T,t,v} \to \Pi_{T,0,v}$.

Now it is easy to check that $h_v(x)= \Erw{e^{i\skalar{v, \sum_{k=1}^\infty M_k \cdots M_1 x}}}$ is the eigenfunction of $T_{0,v}$ corresponding to eigenvalue $k(0,v)=1$, i.e.  $\Pi_{T,0,v} f = c_f \cdot h_v$. Since $\Pi_{T,t,v} h_v \neq 0$ for $t$ small enough {\red and this is an eigenfunction of $T_t{t,v}$ corresponding to the eigenvalue $k(t,v)$,  we have by \eqref{intertwine} that
\begin{equation} \Delta_t \Pi_{T,t,v}h_v = c_t g_{t,v} \label{term:gtv} \end{equation}
for some  $c_t  \neq 0$}. For more details we refer to \cite{BDG2010} and \cite{Mirek2010}.

\begin{lemma}
Under the assumptions of Theorem \ref{theorem:stable:laws}, for all $v \in S$, and $\kappa\in (0,1) \cup (1, 2)$ we have
\begin{align}\label{conv1}
    \lim_{t \to 0} \frac{k(t,v)-1-i\skalar{v,tm_\kappa}}{t^{\kappa}} = \int_{V}\left( \left(e^{i \skalar{v,x}}-1\right) h_v(x) - i\1[(1,2)](\kappa)\skalar{v,x} \right)\Lambda_\kappa(dx).
\end{align}
\end{lemma}

\begin{proof}
Using \eqref{term:gtv} in \eqref{identity:ktv}, yields
\begin{align*} (k(t,v)-1) \nu( \Delta_t \Pi_{T,t,v}h_v)= \int_{\mathbb{R}^d} (e^{it \skalar{v,x}} -1 )  \Delta_t \Pi_{T,t,v}h_v(x) \nu(dx). \end{align*}
Arguing in a similar way as in  (\cite[Section 6]{Mirek2010}), we
obtain that
\begin{equation}\label{reason:for:T} \int_{\R^d} (e^{it\skalar{v, x }}-1) \cdot \bigl(\Delta_t \Pi_{T,t,v} h_v(x) - \Delta_t h_v(x)\bigr) \nu(dx) = o(t^\kappa) \end{equation}
as well as $\bigl(\nu(\Delta_t \Pi_{T,t,v} h_v) -1\bigr) =
o(t^\kappa)$. Thus
\begin{align} \lim_{t \to 0} \frac{k(t,v)-1}{t^\kappa} = & \lim_{t \to 0} t^{-\kappa} \int_{\mathbb{R}^d} (e^{it \skalar{v,x}} -1 )  \Delta_t h_v(x) \nu(dx) \nonumber \\ = & \lim_{t \to 0} t^{-\kappa} \Erw{\bigl(e^{i\skalar{v,tR}} - 1\bigr) h_v(tR) }. \label{lim:1} \end{align}
If $\kappa \in (0,1)$,  checking that $(e^{i\skalar{v,\cdot}} -
1\bigr) h_v$ satisfies \eqref{prop:fkonv} the proof follows. If
$\kappa \in (1,2)$, we write
\begin{align*}&   \int_{\mathbb{R}^d} (e^{it \skalar{v,x}} -1 )  \Delta_t h_v(x)  \nu(dx)  \\
= &  \int_{\mathbb{R}^d} (e^{it \skalar{v,x}} -1 ) ( \Delta_t
h_v(x) -1)  \nu(dx) + \int_{\R^d} (e^{it\skalar{v,x}} -1 -
i\skalar{v, tx}) \nu(dx) + i \skalar{v, tm_\kappa}.
\end{align*} We check that the functions under the integral
satisfy \eqref{prop:fkonv} and then we use \eqref{lim:1}.
 For more details we refer to (\cite[Section 5]{BDG2010}). 
%$$ \lim_{t \to 0} t^{-\kappa} \nu((e^{it\skalar{v, \cdot }}-1) \Delta_t h_v) = \lim_{t \to %0} t^{-\kappa} \Erw{e^{i\skalar{v, tR }}-1)  h_v(tR)} = \int_{\R^d_*} (e^{it\skalar{v,x %}}-1)  h_v(x) \Lambda(dx).$$
%This limit exists since the function in question satisfies
%$$ \sup_{x \in {\R^d_*}} \abs{x}^{-\kappa} \abs{\log \abs{x}}^{1+ \epsilon} \abs{f(x)} < %\infty.$$
\end{proof}
{\green This proves the pointwise convergence of $ \Xi_{\kappa,n}(sv)$ to $\exp \left( s^\kappa C_\kappa(v) \right)$. Continuity  at 0 follows from the dominated convergence theorem.}

\section{The limit measure is nondegenerate}\label{nondeg:section}
In this section we will assume additionally that assumption
\eqref{A4*} is in force. As mentioned in Remark \ref{remark!2},
all results of \cite{AM2010} then hold for the SFPE given by
$(M^\top, Q)$.

We start with a corollary to Proposition \ref{prop!}. Note that
$\sigma$ here is a measure on $S$ seen as set of column vectors.
\begin{corollary} \label{corr!}
If $R$ is the solution to the RDE $R \stackrel{d}= M^\top R +Q$,
then
\begin{equation} \label{K2} 0 < \int_{S} \Erw{(\skalar{R,y}^+)^{\kappa}-(\skalar{M^\top R,y}^+)^{\kappa}}\sigma(dy) < \infty, \end{equation}
where $\sigma$ is (up to scalar multiplication) the unique
solution to $\sigma T_\kappa^* = \sigma$, with $T_\kappa^*$ as
defined in \eqref{Tkappa2}. \end{corollary}

\begin{proof} Given \eqref{A4*}, Proposition \ref{prop!} holds also for the RDE given by $(M^\top, Q)$, mutatis mutandis.
Observe that considering $T_\kappa^*$ on functions of column vectors is equivalent to considering the operator $\hat{T}_\kappa$
$$ \hat{T}_\kappa f(v) = \Erw{f\bigl( (vM^\top)^\sim\bigr) |vM^\top|^\kappa}$$
on functions of row vectors. But $\hat{T_\kappa}$ is the right expression for $T_\kappa$ of \eqref{Tkappa0} for the RDE given by $(M^\top, Q)$. In particular the only eigenvalue of $T_\kappa^*$ of modulus one  is 1, and the corresponding eigenspace is one-dimensional. Finally use that $r$ is positive, and that $r(v)^{-1} \pi(dv) = \sigma (dv)$. 
\end{proof}

The following nice observation is the main ingredient in the proof of nondegeneracy:

\begin{corollary} \label{corr!2} For $\sigma_\kappa$ as defined in Theorem \ref{theorem:vague:convergence}, for $R$ being solution to $R\stackrel{d}= M^\top R +Q$,
\begin{equation} \label{nondeg0}
 0 < \int_{S} \Erw{(\skalar{M^\top R+Q,y}^+)^{\kappa}-(\skalar{M^\top R,y}^+)^{\kappa}}\sigma_\kappa(dy) < \infty
\end{equation}
\end{corollary}

\begin{proof}
Due to Theorem \ref{theorem:vague:convergence}, $\sigma_\kappa T_\kappa^* = \sigma_\kappa$. But if \eqref{A4*} holds, this identifies $\sigma_\kappa$ up to scalar multiples, so $\sigma$ of Corollary \ref{corr!} is a scalar multiple of $\sigma_\kappa$. One may replace $R$ by $M^\top R+Q$, since
$$ \Erw{(\skalar{R,y}^+)^{\kappa}-(\skalar{M^\top R,y}^+)^{\kappa}} = \int_0^\infty u^{\kappa-1} \left[ \P{\skalar{R,y} >u} - \P{\skalar{M^\top R, y} >u} \right] du ,$$
see \cite[Lemma 9.4]{Goldie1991}. 
\end{proof}

\begin{lemma}
Under the assumptions of Theorem \ref{theorem:stable:laws} and with \eqref{A4*} in force, the limit laws are fully nondegenerate, i.e. for all $v \in S^\top$ $\Re C_{\kappa}(v)<0$.
\end{lemma}
\begin{proof}
Notice that
\begin{eqnarray}
\Re C_{\kappa}(v)&=& \Re \bigg(
\int_{V}\big( e^{i \skalar {v,x}} - 1 \big)\E\big[ e^{i\skalar{v,{W(x)}}} \big] \Lambda_\kappa(dx)
\bigg)\nonumber\\
&=& \int_0^{\infty}
\int_{S}\mathbb{\mathbb{E}}\Big[ \cos\big( s\skalar{v,{W(w)+w}} \big)
- \cos\big( s\skalar{v,{W(w)}} \big)
\Big] \sigma_\kappa(dw)\frac {ds}{s^{\kappa+1}} \nonumber \\
 &=& C(\kappa)\cdot \int_{S} \mathbb{E}\Big[ \big| \skalar{ v,{W(w)+w}} \big|^{\kappa} - \big| \skalar{v,{W(w)}}
\big|^{\kappa} \Big]\sigma_\kappa(dw), \label{nondeg1}\end{eqnarray}
 for $C(\kappa) = \int_0^{\infty} \frac{\cos s-1}{s^{\kappa+1}}<0$. Let $W^*v:=\sum_{k=1}^\infty M_1^\top \cdots M_k^\top v$. Then it suffices to prove that
\begin{equation} \int_{S} \mathbb{E}\Big[ \big| \skalar{ W^*v + v,w} \big|^{\kappa} - \big| \skalar{W^*v,w}
\big|^{\kappa} \Big]\sigma_\kappa(dw) \label{nondeg2}\end{equation}
is positive.

Therefore, we use that $(M^\top, v)$ satisifies the assumptions of Theorem \ref{theorem:old}, and that $W^*v+v$ is a solution
of the random difference equation
\begin{equation}
\label{rde}
W^*v+v \stackrel{d}= M^\top (W^*v+v) + v.
\end{equation}

% To be precise and have the form $MR+Q$, $MR$ in the application of the corrolary, one could consider $W_0v = \sum_{k=2}^\infty M_2^\top \cdots M_k^\top v$, which has the same distribution as $W^*v$, and satisfies $$ W^*v +v = M_1^\top W_0 v + v, \quad W^*v = M_1^\top W_0.$$

In this situation, Corollary \ref{corr!2} yields
\begin{equation} \label{nondeg3}
0 < \int_S \Erw{  \big( \skalar{ W^*v + v,w}^+ \big)^{\kappa} - \big( \skalar{W^*v,w}^+
\big)^{\kappa} } \sigma_\kappa(dw) < \infty,
\end{equation}
and, considering $-v$ instead of $v$ (notice that $\sigma_\kappa$ does not depend on $v$)
\begin{equation} \label{nondeg4}
0 < \int_S \Erw{  \big( \skalar{ W^*v + v,w}^- \big)^{\kappa} - \big( \skalar{W^*v,w}^-
\big)^{\kappa} } \sigma_\kappa(dw) < \infty.
\end{equation}
Since both integrals are finite, adding \eqref{nondeg3} and \eqref{nondeg4} implies that \eqref{nondeg2} is indeed positive. 
\end{proof}

\textsc{Acknowledgements}\\The major part of this work was done during two visits of the second author at the University of Wroclaw, Institute of Mathematics to which he wishes to express his gratitude for hospitality and a stimulating atmosphere. \\

{\blue The authors are grateful to the referee for a very careful reading
of the manuscript and useful remarks that lead to the improvement
of the presentation.}

\nocite{BDGHU2009,GL2008,GL2012,GGL2011}
%\bibliographystyle{abbrv}
%\bibliography{StableLaws}

\end{document}